\documentclass[letterpaper,11pt,reqno]{amsart}

\makeatletter
\usepackage{amssymb}
\usepackage{latexsym}
\usepackage{amsbsy}
\usepackage{amsfonts}
\usepackage{hyperref}
\usepackage{graphicx}
\usepackage{enumerate}
\usepackage{enumitem}
\usepackage{scalerel}
\usepackage{color}
\usepackage[round]{natbib}   

\usepackage{tikz}

\def\marginpar#1{\ignorespaces}

\DeclareMathOperator\vol{Vol}

\newtheorem{theorem}{Theorem}
\newtheorem*{theorem*}{Main Results}

\newtheorem{assumption}{Assumption}
\newtheorem{remark}{Remark}

\newtheorem{lemma}{Lemma}
\makeatother
\begin{document}
\title[Simulated annealing]{Simulated annealing from continuum to discretization: a convergence analysis via the Eyring--Kramers law}

\author[Wenpin Tang]{{Wenpin} Tang}
\address{Department of Industrial Engineering and Operations Research, Columbia University.
} \email{wt2319@columbia.edu}

\author[Xun Yu Zhou]{{Xun Yu} Zhou}
\address{Department of Industrial Engineering and Operations Research, Columbia University.}
\email{xz2574@columbia.edu}

\date{\today}
\begin{abstract}
We study the convergence rate of continuous-time simulated annealing $(X_t; \, t \ge 0)$
and its discretization $(x_k; \, k =0,1, \ldots)$ for approximating the global optimum of a given function $f$.
We prove that the tail probability $\mathbb{P}(f(X_t) > \min f +\delta)$ (resp. $\mathbb{P}(f(x_k) > \min f +\delta)$) decays polynomial in time (resp. in cumulative step size), and provide an explicit rate as a function of the model parameters.
Our argument applies the recent development on functional inequalities for the Gibbs measure at low temperatures -- the Eyring-Kramers law.
In the discrete setting, we obtain a condition on the step size to ensure the convergence.
\end{abstract}

\maketitle

\textit{Key words:} Simulated annealing, convergence rate, Euler discretization, Eyring-Kramers law, functional inequalities, overdamped Langevin equation.

\section{Introduction}

\quad Simulated annealing (SA) is an umbrella term which denotes a set of stochastic optimization methods.
The goal of SA is to find the global minimum of a function $f: \mathbb{R}^d \to \mathbb{R}$, in particular when $f$ is non-convex.
These methods have many applications in physics, operations research and computer science,
see e.g. \cite{VA87, KAJ94, DCM19}.
The name is inspired from annealing in metallurgy, which is a process aiming to increase the size of crystals by heating and controlled cooling.
The stochastic version of SA was independently proposed by \cite{KGV83} and \cite{Cer85}.
The idea is as follows: consider a stochastic process related to $f$ which is subject to thermal noise.
When simulating this process, one decreases the temperature slowly over the time.
As a result, the stochastic process escapes from saddle points and local optima, and converges to the global minimum of $f$ with high probability.
This is generally true if the cooling is slow enough, and it is important to find the right stochastic process with the fastest possible cooling schedule which approximates the global optimum.

\quad In this paper, we explore the convergence rate of continuous-time SA and its discretization.
To be more precise, define the {\em continuous-time SA process} $(X_t; \, t \ge 0)$ by
\begin{equation}
\label{eq:SA}
dX_t = - \nabla f(X_t) dt + \sqrt{2 \tau_t} \, dB_t, \quad X_0 \stackrel{d}{=} \mu_0(dx),
\end{equation}
where $(B_t; \, t \ge 0)$ is $d$-dimensional Brownian motion, $\tau_t$ is the cooling schedule of temperature, and $\mu_0(dx)$ is some initial distribution.
This formulation was first considered by \cite{GH86}.
If $\tau_t \equiv \tau$ is constant in time, the process \eqref{eq:SA} is the well-known {\em overdamped Langevin equation} whose stationary distribution is the Gibbs measure $\nu_{\tau}(dx) \propto \exp(-f(x)/\tau)dx$.
Thus, we sometimes  call the process \eqref{eq:SA} an {\em SA adapted overdamped Langevin equation} as well.
See Section \ref{sc3} for background.
We also consider the Euler-Maruyama discretization of the continuous-time simulated annealing process.
Let $\eta_k$ be the step size at iteration $k$, and $\Theta_k:=\sum_{j \le k} \eta_j$ be the cumulative step size up to iteration $k$.
The {\em discrete SA process} $(x_k; \, k = 0,1, \ldots)$ is defined by
\begin{equation}
\label{eq:discreteSA}
x_{k+1} = x_k - \nabla f(x_k) \eta_k + \sqrt{2 \tau_{\scaleto{\Theta_k}{5 pt}} \eta_k} Z_k, \quad x_0 \stackrel{d}{=} \mu_0(dx),
\end{equation}
where $(Z_k; \, k = 0, 1, \ldots)$ are independent and identically distributed standard normal vectors,
and $\tau_{\scaleto{\Theta_k}{5 pt}}$ is the cooling schedule at iteration $k$.
For $\tau_t \equiv \tau$ constant in time, the scheme \eqref{eq:discreteSA} is known as the
{\em unadjusted Langevin algorithm} (ULA) which approximates the Gibbs measure $\nu_{\tau}$.
The algorithm was introduced in \cite{Parisi81, GM94}, and further studied by \cite{RT96, DM17}.

\quad The goal of this paper is to study the decay in time of the tail probability
\begin{equation*}
\mathbb{P}(f(X_t) > \min f +\delta) \quad \mbox{or} \quad \mathbb{P}(f(x_k) > \min f+ \delta),
\end{equation*}
under suitable conditions on the function $f$, the cooling schedule $\tau_t$, and the discretization scheme $\eta_k, \Theta_k$.
Let us mention a few motivations.
First there are a line of works on the interplay between sampling and optimization (\cite{RRT17, MCC19, MCJ19}).
Note that if $\tau_t \equiv \tau$ is constant in time, the overdamped Langevin equation converges to the Gibbs measure
$\nu_{\tau}$,
and for $\tau$ sufficiently small, the Gibbs measure $\nu_{\tau}$ approximates the Dirac mass at the global minimum of $f$.
Combining these two ingredients, SA sets the cooling schedule $\tau_t$ to decrease to $0$ over the time.
It is then expected that the SA process \eqref{eq:SA} or \eqref{eq:discreteSA} converges to the global minimum as $t \to \infty$ or $k \to \infty$.
There are also recent works on various noisy gradient-based algorithms (\cite{GHJY15, JG17, CDT20, GHT20}), which aims to escape saddle points and find a local minimum of $f$.
In comparison with these methods, SA has priority to find the global minimum but at the cost of longer exploration time.

\quad The main tool in our analysis is the Eyring-Kramers law, which is a set of functional inequalities for the Gibbs measure at low temperatures.
Let us explain our results.
It was shown in \cite{GH86, CH87} that the correct order of $\tau_t$ for the process \eqref{eq:SA} to converge to the global minimum of $f$ is $(\ln t)^{-1}$.
In fact, there is a phase transition related to the {\em critical depth} $E_{*}$ of the function $f$:
\begin{enumerate}[itemsep = 3 pt]
\item[(a)]
If $\tau_t \le \frac{E}{\ln t}$ for $t$ large enough with $E < E_{*}$, then
$\limsup_{t \to \infty} \mathbb{P}(f(X_t) \le \min f + \delta) < 1$.
\item[(b)]
If $\tau_t \ge \frac{E}{\ln t}$ for $t$ large enough with $E > E_{*}$, then
$\lim_{t \to \infty} \mathbb{P}(f(X_t) \le \min f + \delta) = 1$.
\end{enumerate}
Roughly speaking, the critical depth $E_{*}$ is the largest hill one needs to climb starting from a local minimum to the global minimum.
The formal definition of the critical depth $E_{*}$ will be given in Assumption \ref{assump:nondeg}, but 
see Figure \ref{fig:CD} below for an illustration when $f$ is a double-well function.
\begin{figure}[h]
\centering
\includegraphics[width=0.6\columnwidth]{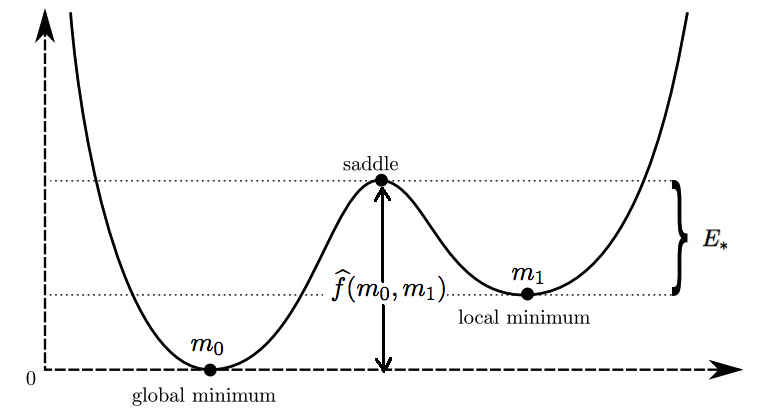}
\caption{Illustration of the critical depth of a double-well function.}
\label{fig:CD}
\end{figure}
Part $(a)$ was proved by \cite{HKS89} using a sophisticated argument that involves the Poincar\'e inequality.
Part $(b)$ was proved by \cite{Miclo92} who characterized the fastest cooling schedule by the Eying-Kramers law for the log-Sobolev inequality.
See also \cite{Miclo95, Zitt08, FT20} for similar results under different conditions on the function $f$.
The convergence rate of $(f(X_t); \, t \ge 0)$ to the global minimum of $f$ was considered by \cite{Mar97} only for $\delta$ sufficiently small.
However, the rate of $\mathbb{P}(f(X_t) > \min f +\delta)$  for general $\delta > 0$ has yet been studied since no estimates of the log-Sobolev inequality for the Gibbs measure at low temperatures were known until the mid-2010s. 
Taking advantage of recently developed theory (\cite{MS14, MSTW18, VW19}), we are able to give a non-asymptotic convergence rate of both continuous-time and discrete SA.

\quad To simplify the notation, we assume throughout this paper that
\begin{equation*}
\min_{\mathbb{R}^d} f(x) = 0,
\end{equation*}
i.e. the global minimum of $f$ is $0$ by considering $f - \min f$.
Our results are outlined as follows, and the precise statement of these results will be given in Section \ref{sc2}.
\begin{theorem*}[Informal]
Under some assumptions on the function $f$, we have:
\begin{enumerate}[itemsep = 3 pt]
\item[(i)]
Assume that  $\tau_t \ge \frac{E}{\ln t}$ for $t$ large enough with $E > E_{*}$.
For $\delta > 0$,
there exists $C > 0$ (depending on $\delta, f, E, d$) such that
\begin{equation*}
\mathbb{P}(f(X_t) > \delta) \le C t^{-\min\left(\frac{\delta}{E}, \frac{1}{2}(1 - \frac{E_{*}}{E}) \right)}.
\end{equation*}
\item[(ii)]
Assume that  $\tau_t \ge \frac{E}{\ln t}$ for $t$ large enough with $E > E_{*}$.
Also assume that $\Theta_k \to \infty$, $\eta_{k+1} \Theta_k \to 0$ and
$\frac{\ln \Theta_k}{\sum_{j \le k} \eta_{j+1} \Theta_j^{-E_{*}/E}} \to 0$ as $k \to \infty$.
For $\delta > 0$,
there exists $C > 0$ (depending on $\delta, f, E, d$) such that
\begin{equation*}
\mathbb{P}(f(x_k) > \delta) \le C \Theta_k^{-\min\left(\frac{\delta}{E}, \frac{1}{2}(1 - \frac{E_{*}}{E}) \right)}.
\end{equation*}
\end{enumerate}
\end{theorem*}

\quad The contribution of this paper is twofold:
\begin{itemize}[itemsep = 3 pt]
\item
{\bf Polynomial decay rate}.
We prove that the tail probability $\mathbb{P}(f(X_t) > \delta)$ (resp. $\mathbb{P}(f(x_k) > \delta)$) decays polynomially in time (resp. in cumulative step size).
In the continuous setting, \cite{M18} obtained the same rate of convergence for SA adapted underdamped Langevin equation, and \cite{MSTW18} considered an improvement of SA via parallel tempering.
However, the convergence rate for continuous-time SA, i.e. part $(i)$ has yet been recorded in the literature
though this result is probably understood in the Eyring-Kramers folklore.
We provide a self-contained treatment which bridges the literature and, more importantly, can be applied to obtain the new result in the discrete setting.
\item
{\bf Choice of step size}.
Part $(ii)$ for the discrete simulated annealing is completely novel to our best knowledge, and it also gives a practical guidance on the choice of step size for discretization.
The condition $\Theta_k \to \infty$ indicates that the step size cannot be chosen too small, while the condition $\eta_{k+1} \Theta_k \to 0$ suggests that the step size cannot be chosen too large.
The condition $\frac{\ln \Theta_k}{\sum_{j \le k} \eta_{j+1} \Theta_j^{-E_{*}/E}} \to 0$ is purely technical.
For instance, $\eta_{k} = k^{-\theta}$ with $\theta \in (\frac{1}{2}, 1]$ satisfies the conditions in $(ii)$ to ensure the convergence.
\end{itemize}
Also note that the rate $\min\left(\frac{\delta}{E}, \frac{1}{2}(1 - \frac{E_{*}}{E}) \right)$ is smaller than $\frac{1}{2}$.
It is interesting to know whether this rate is optimal, and we leave the problem for future work.

\quad There is another interesting problem: the dependence of the constant $C$ on the dimension $d$.
The issue is subtle, since most analysis including the Eyring-Kramers law uses Laplace's method.
However, Laplace's method may fail if both the dimension $d$ and the inverse temperature $1/\tau$ tend to infinity (\cite{SM95}).
As shown in Remark \ref{rk:Laplace}, we obtain an upper bound for $C$ which is exponential in $d$.
This suggests the convergence rate is exponentially slow as the dimension increases, which concurs with the fact that
finding the global minimum of a general nonconvex function is NP-hard (\cite{JK17}).

\quad Finally, we mention a few approaches to accelerate or improve SA.
\cite{FQG97} considered a cooling schedule depending on both time and state;
\cite{Pav07} used the L\'evy flight;
\cite{M18} studied SA adapted to underdamped Langevin equation;
\cite{MSTW18} applied the replica exchange technique;
\cite{GXZ20} employed a relaxed stochastic control formulation, originally proposed by \cite{WZZ} for reinforcement learning, to derive a state-dependent temperature control schedule. 

\medskip
The remainder of the paper is organized as follows. 
Section \ref{sc2} presents the assumptions and our main results.
Section \ref{sc3} provides background on diffusion processes and functional inequalities.
The result for the continuous-time simulated annealing (Theorem \ref{thm:contrate}) is proved in Section \ref{sc4}.
The result for the discrete simulated annealing (Theorem \ref{thm:discreterate}) is proved in Section \ref{sc5}.
We conclude with Section \ref{sc6}.

\section{Main results}
\label{sc2}

\quad In this section, we make precise  the informal statement in the introduction, and present the main results of the paper.
We first collect the notations that will be used throughout this paper.
\begin{itemize}[label = {--}, itemsep = 3 pt]
\item
The notation $| \cdot |$ is the Euclidean norm of a vector, and $a \cdot b$ is the scalar product of vectors $a$ and $b$.
\item
For a function $f: \mathbb{R}^d \to \mathbb{R}$,
let $\nabla f$, $\nabla^2 f$ and $\Delta f$ denote its gradient, Hessian and Laplacian respectively.
\item
For $X, Y$ two random variables, $||X - Y||_{TV}$ denotes the total variation norm of the signed measure corresponding to $X - Y$.
\item
The symbol $a \sim b$ means that $a/b \to 1$ as some problem parameter tends to $0$ or $\infty$.
Similarly, the symbol $a = \mathcal{O}(b)$ means that $a/b$ is bounded as some problem parameter tends to $0$ or $\infty$.
\end{itemize}
We use $C$ for a generic constant which depends on problem parameters ($\delta, f, E \ldots$), and may change from line to line.

\quad Next, we present a few assumptions on the function $f$.
These assumptions are standard in the study of metastability.

\begin{assumption}
\label{assump:reggrow}
Let $f: \mathbb{R}^d \to \mathbb{R}$ be smooth, bounded from below,
and satisfy the conditions:
\begin{enumerate}[itemsep = 3 pt]
\item[(i)]
$f$ is non-degenerate on the set of critical points. That is, for some $C > 0$,
\begin{equation*}
\frac{|\xi|}{C} \le |\nabla^2f(x) \xi| \le C|\xi| \quad \mbox{for each } x \in \{z: \nabla f(z) = 0\}.
\end{equation*}
\item[(ii)]
There exists $C > 0$ such that
\begin{equation*}
\liminf_{|x| \to \infty} \frac{|\nabla f(x)|^2 - \Delta f(x)}{|x|^2} \ge C, \quad \inf_{x} \nabla^2f(x) \ge -C.
\end{equation*}
\end{enumerate}
\end{assumption}

\quad Let us make a few comments on Assumption \ref{assump:reggrow}.
The condition $(ii)$ implies that $f$ has at least quadratic growth at infinity.
This is a necessary and sufficient condition to obtain the log-Sobolev inequality (see \cite[Theorem 3.1.21]{Royer07} and
Section \ref{sc32}) which is key to convergence analysis.
The conditions $(i)$ and $(ii)$ imply that the set of critical points is discrete and finite \cite[Remark 1.6]{MS14}.
In particular, it follows that the set of local minimum points $\{m_1, \ldots, m_N\}$ is also finite, with $N$ the number of local minimum points of $f$.

\quad To keep the presentation simple, we make additional assumptions on $f$, following \cite[Assumption 2.5]{MSTW18}.
Define the saddle height $\widehat{f}(m_i,m_j)$ between two local minimum points $m_i, m_j$ by
\begin{equation}
\widehat{f}(m_i,m_j) : = \inf \left\{\max_{s \in [0,1]} f(\gamma(s)): \gamma \in \mathcal{C}[0,1], \, \gamma(0) = m_i, \, \gamma(1) = m_j \right\}.
\end{equation}
See Figure \ref{fig:CD} for an illustration of the saddle height $\widehat{f}(m_0,m_1)$ when $f$ is a double-well function with $m_0$ the global minimum and $m_1$ the local minimum.
\begin{assumption}
\label{assump:nondeg}
Let $m_1, \cdots, m_N$ be the positions of the local minima of $f$.
\begin{enumerate}[label=(\roman*), itemsep = 3 pt]
\item
$m_1$ is the unique global minimum point of $f$, and $m_1, \ldots, m_N$ are ordered in the sense that there exists $\delta > 0$ such that
\begin{equation*}
f(m_N) \geq f(m_{N-1}) \geq \cdots \geq f(m_2) \ge \delta \quad \mbox{and} \quad f(m_1) = 0.
\end{equation*}
\item
For each $i, j \in \{1, \ldots, N\}$, the saddle height between $m_i, m_j$ is attained at a unique critical point $s_{ij}$ of index one. That is, $f(s_{ij}) = \widehat{f}(m_i,m_j)$, and if $\{\lambda_1, \ldots, \lambda_n\}$ are the eigenvalues of $\nabla^2 f(s_{ij})$, then $\lambda_1< 0$ and $\lambda_i > 0$ for $i \in \{2, \ldots, n\}$.
The point $s_{ij}$ is called the communicating saddle point between the minima $m_i$ and $m_j$.
\item
There exists $p \in [N]$ such that the energy barrier $f(s_{p1}) - f(m_p)$ dominates all the others. That is,
there exists $\delta > 0$ such that for all $i \in [N] \setminus \{p\}$,
\begin{equation*}
E_* := f(s_{p1}) - f(m_p) \ge f(s_{i1}) - f(m_i) + \delta.
\end{equation*}
The dominating energy barrier $E_*$ is called the critical depth.
\end{enumerate}
\end{assumption}

\quad The convergence result for the continuous-time SA \eqref{eq:SA} is stated as follows.
The proof will be given in Section \ref{sc4}.
\begin{theorem}
\label{thm:contrate}
Let $f$ satisfy Assumption \ref{assump:reggrow} $\&$ \ref{assump:nondeg}.
Assume that $\tau_t \sim \frac{E}{\ln t}$ and $\frac{d}{dt}\left(\frac{1}{\tau_t}\right) = \mathcal{O}\left(\frac{1}{t}\right)$
as $t \to \infty$ with $E > E_{*}$.
Also assume the moment condition for the initial distribution $\mu_0$: for each $p \ge 1$,
there exists $C_p > 0$ such that
\begin{equation}
\label{eq:moment}
\int_{\mathbb{R}^d} f(x)^p \mu_0(dx) \le C_p.
\end{equation}
Then for each $\delta, \varepsilon > 0$, there exists $C > 0$ (depending on $\delta, \varepsilon, f, \mu_0, E, d$) such that
\begin{equation}
\label{eq:main1}
\mathbb{P}(f(X_t) > \delta) \le C t^{-\min\left(\frac{\delta}{E}, \frac{1}{2}(1 - \frac{E_{*}}{E}) \right) + \varepsilon}.
\end{equation}
\end{theorem}

\quad To get the convergence result for the discrete simulated annealing, we need an additional condition on the function $f$.
\begin{assumption}
\label{assump:upper}
The gradient $\nabla f$ is globally Lipschitz, i.e. $|\nabla f(x) - \nabla f(y)| \le L |x - y|$ for some $L > 0$.
\end{assumption}

\quad The convergence result for the discrete simulated annealing \eqref{eq:discreteSA} is stated as follows.
The proof will be given in Section \ref{sc5}.
\begin{theorem}
\label{thm:discreterate}
Let $f$ satisfy Assumption \ref{assump:reggrow}, \ref{assump:nondeg} $\&$ \ref{assump:upper}, and let the condition \eqref{eq:moment} for $\mu_0$ hold.
Assume that $\tau_t \sim \frac{E}{\ln t}$ and $\frac{d}{dt}\left(\frac{1}{\tau_t}\right) = \mathcal{O}\left(\frac{1}{t}\right)$ as $t \to \infty$ with $E > E_{*}$.
Also assume that $\Theta_k \to \infty$,
\begin{equation}
\label{eq:dominate}
\eta_{k+1}\Theta_k \to 0,
\end{equation}
and
\begin{equation}
\label{eq:dominate2}
\frac{\ln \Theta_k}{\sum_{j \le k} \eta_{j+1} \Theta_j^{-E_{*}/E}} \to 0,
\end{equation}
as $k \to \infty$.
Then for each $\delta, \varepsilon > 0$, there exists $C > 0$ (depending on $\delta, \varepsilon, f, \mu_0, E, d$) such that
\begin{equation}
\label{eq:main2}
\mathbb{P}(f(x_k) > \delta) \le C \Theta_k^{-\min\left(\frac{\delta}{E}, \frac{1}{2}(1 - \frac{E_{*}}{E}) \right) + \varepsilon}.
\end{equation}
\end{theorem}

\section{Preliminaries}
\label{sc3}

\quad In this section, we present a few vocabularies and basic results of diffusion processes and functional inequalities.
We also explain how these results are applied in the setting of SA, which will be useful in our convergence analysis.

\subsection{Diffusion processes and SA}
\label{sc31}
Consider the general diffusion process $(X_t; \, t \ge 0)$ in $\mathbb{R}^d$ of form:
\begin{equation}
\label{eq:diffusion}
dX_t = b(t, X_t) dt + \sigma(t, X_t) dB_t, \quad X_0 \stackrel{d}{=} \mu_0(dx),
\end{equation}
where $(B_t; t \ge 0)$ is a $d$-dimensional Brownian motion,
with the drift $b: \mathbb{R}_{+} \times \mathbb{R}^d \to \mathbb{R}^d$ and the covariance matrix $\sigma: \mathbb{R}_{+} \times \mathbb{R}^d \to \mathbb{R}^{d \times d}$.
To ensure the well-posedness of the SDE \eqref{eq:diffusion}, it requires some growth and regularity conditions on $b$ and $\sigma$.
For instance,
\begin{itemize}[itemsep = 3 pt]
\item
If $b$ and $\sigma$ are Lipschitz and have linear growth in $x$ uniformly in $t$, then \eqref{eq:diffusion} has a strong solution which is pathwise unique.
\item
If $b$ is bounded, and $\sigma$ is bounded, continuous and strictly elliptic, then \eqref{eq:diffusion} has a weak solution which is unique in distribution.
\end{itemize}
We refer to \cite{SV79, RW87} for background and further developments on the well-posedness of SDEs, and to
\cite[Chapter 1]{CE05} for a review of related results.

\quad Another important aspect is the distributional property of $(X_t; \, t \ge 0)$ governed by the SDE \eqref{eq:diffusion}.
Let $\mathcal{L}$ be the infinitesimal generator of the diffusion process $X$ defined by
\begin{align}
\label{eq:infg}
\mathcal{L}g(t,x) &: = b(t,x) \cdot \nabla g(x) + \frac{1}{2} \sigma(t,x) \sigma(t,x)^T:\nabla^2 g(x) \notag\\
&: = \sum_{i=1}^d b_i(t,x) \frac{\partial}{\partial x_i} g(x) + \frac{1}{2} \sum_{i,j = 1}^d \left( \sigma(t,x) \sigma(t,x)^T \right)_{ij} \frac{\partial^2}{\partial x_i \partial x_j} g(x),
\end{align}
and $\mathcal{L}^{*}$ be the corresponding adjoint operator given by
\begin{align}
\label{eq:adjoint}
\mathcal{L}^{*}g(t,x) & := -\nabla \cdot (b(t,x) g(x)) + \frac{1}{2} \nabla^2:(\sigma(t,x) \sigma(t,x)^T g(x)) \notag \\
&:= -\sum_{i = 1}^d \frac{\partial}{\partial x_i} (b_i(t,x) g(x)) + \frac{1}{2} \sum_{i,j = 1}^d \frac{\partial^2}{\partial x_i \partial x_j} (\sigma(t,x) \sigma(t,x)^T g(x))_{ij},
\end{align}
where $g: \mathbb{R}^d \to \mathbb{R}$ is a suitably smooth test function,
and $:$ denotes the Frobenius inner product which is the component-wise inner product of two matrices.
 The probability density $\rho_t(\cdot)$ of the process $X$ at time $t$ then satisfies the Fokker-Planck equation:
 \begin{equation}
 \label{eq:FPdiffusion}
 \frac{\partial \rho_t}{\partial t} = \mathcal{L}^{*} \rho_t.
 \end{equation}
Specializing \eqref{eq:FPdiffusion} to the SA process \eqref{eq:SA} with $b(t,x) = -\nabla f(x)$ and $\sigma(t,x)= \sqrt{2 \tau_t} \, I_d$, we have that the probability density $\mu_t(\cdot)$ of $X$ governed by the SDE \eqref{eq:SA} satisfies
\begin{equation}
\label{eq:FPSA}
\frac{\partial \mu_t}{\partial t}  = \nabla \cdot (\mu_t \nabla f) + \tau_t \Delta \mu_t.
\end{equation}
Under further growth conditions on $b$ and $\sigma$,  it can be shown that as $t\to \infty$, 
$\rho_t(\cdot) \to \rho_{\infty}(\cdot)$  which is the stationary distribution of $(X_t; \, t \ge 0)$.
It is easily deduced from \eqref{eq:FPdiffusion} that $\rho_{\infty}$ is characterized by the equation $\mathcal{L}^{*} \rho_{\infty} = 0$; 
see \cite{EK86, MT933} for a general theory on stability of diffusion processes, and \cite[Section 2]{Tang19} for a summary with various pointers to the literature.

\quad However, for general $b$ and $\sigma$, the stationary distribution $\rho_{\infty}(\cdot)$ does not have a closed-form expression.
One good exception is $b(t,x) = - \nabla f(x)$ and $\sigma(t,x) = \sqrt{2 \tau} \, I_d$, where $X$ is governed by the overdamped Langevin equation:
\begin{equation}
\label{eq:overdamped}
dX_t = - \nabla f(X_t) dt + \sqrt{2 \tau} \, dB_t, \quad X_0 \stackrel{d}{=} \mu_0(dx).
\end{equation}
Such a process is time-reversible, and
the stationary distribution, under some growth condition on $f$, is the Gibbs measure
\begin{equation}
\label{eq:Gibbs}
\nu_{\tau}(dx) = \frac{1}{Z_{\tau}} \exp \left(-\frac{f(x)}{\tau} \right)dx,
\end{equation}
where $Z_{\tau}:= \int_{\mathbb{R}^d} \exp(-f(x)/\tau)dx$ is the normalizing constant.
Much is known about the overdamped Langevin dynamics.
For instance, if $f$ is $\lambda$-convex (i.e. $\nabla^2 f + \lambda I_d$ is positive definite),
the overdamped Langevin process \eqref{eq:overdamped} converges exponentially in the Wasserstein metric with rate $\lambda$ to the Gibbs measure $\nu_{\tau}$  \cite[]{CMV06}.
See also \cite{BGL14} for modern techniques to analyze the evolution of the overdamped Langevin equation and generalizations.

\quad Now we turn to the SA process \eqref{eq:SA}.
The difference between the overdamped Langevin process \eqref{eq:overdamped} and the process \eqref{eq:SA} is that the temperature $\tau_t$ of the latter is decreasing in time.
Due to the time dependence, the limiting distribution of SA is unknown.
As we will see in Section \ref{sc4}, the idea is to approximate \eqref{eq:SA} by a process of Gibbs measures with temperature $\tau_t$.
Since $\tau_t$ decreases to $0$ in the limit, the problem boils down to studying Gibbs measures at low temperatures.
In the next section, we recall some results of Gibbs measures at low temperatures,
which are motivated by applications in molecular dynamics and Bayesian statistics.

\subsection{Functional inequalities and the Erying-Kramers law}
\label{sc32}
Here we present functional inequalities of Gibbs measures at low temperatures $(\tau \to 0)$.
Let $\mu$ and $\nu$ be two probability measures on $\mathbb{R}^d$ such that $\mu$ is absolutely continuous relative to $\nu$, with $d\mu/d\nu$ the Radon-Nikodym derivative.
Define the relative entropy or KL-divergence $H(\mu|\nu)$ of $\mu$ with respect to $\nu$ by
\begin{equation}
\label{eq:KL}
H(\mu|\nu):= \int \log\bigg( \frac{d\mu}{d\nu} \bigg) d\mu = \int \frac{d\mu}{d\nu} \log\bigg( \frac{d\mu}{d\nu} \bigg)d \nu,
\end{equation}
and the Fisher information $I(\mu|\nu)$ of $\mu$ with respect to $\nu$ by
\begin{equation}
\label{eq:FI}
I(\mu|\nu):= \frac{1}{2}\int \bigg|\nabla \bigg(\frac{d\mu}{d\nu}\bigg) \bigg|^2 \bigg( \frac{d \mu}{d \nu}\bigg)^{-1} d\nu.
\end{equation}
We say that the probability measure $\nu$ satisfies the log-Sobolev inequality (LSI) with constant $\alpha > 0$,
if for all probability measures $\mu$ with $I(\mu|\nu) < \infty$,
\begin{equation}
H(\mu|\nu) \le \frac{1}{\alpha} I(\mu|\nu).
\end{equation}
The constant $\alpha$ is called the LSI constant for the probability measure $\nu$.
For instance, the LSI constant $\alpha = 1$ for $\nu$ the multivariate Gaussian with mean $0$ and covariance matrix $I_d$.

\quad The LSI also plays an important role in studying the convergence rate of the overdamped Langevin equation.
Recall that $\nu_{\tau}$ is the Gibbs measure defined by \eqref{eq:Gibbs}, and assume that $\nu_{\tau}$ satisfies the LSI with constant $\alpha_{\tau} > 0$.
It follows from \cite[Theorem 1.7]{Sch12} that by letting $\mu_{\tau, t}$ be the probability distribution of $X_t$ defined by \eqref{eq:overdamped},
we have
\begin{equation}
\label{eq:entropydecay}
H(\mu_{\tau,t}|\nu_{\tau}) \le e^{-2 \tau \alpha_{\tau} t}H(\mu_{\tau, 0}|\nu_{\tau}).
\end{equation}
So larger the value of $\alpha_{\tau}$ is, faster the convergence of the overdamped Langevin process in the KL divergence is.
The subscript `$\tau$' in $\alpha_{\tau}$ suggests the dependence of the LSI constant on the temperature $\tau$, and we are interested in the asymptotics of $\alpha_{\tau}$ at low temperatures as $\tau \to 0$.
This problem was considered by \cite[Corollary 3.18]{MS14}, who derived a sharp lower bound for $\alpha_{\tau}$ as $\tau \to 0$.
\begin{lemma}
\label{fact:EK}
Let $f$ satisfy Assumption \ref{assump:reggrow} $\&$ \ref{assump:nondeg}. Then
the Gibbs measure $\nu_{\tau}$ defined by \eqref{eq:Gibbs} satisfies the LSI with constant $\alpha_{\tau} > 0$ such that
\begin{equation}
\label{eq:EKformula}
\alpha_{\tau} \sim C \exp\left( -\frac{E_{*}}{\tau}\right) \quad \mbox{as } \tau \to 0,
\end{equation}
where $C > 0$ depends on $f, d$.
\end{lemma}

\quad The Eyring-Kramers law provides an estimate on the spectral gap of the overdamped Langevin equation \eqref{eq:overdamped}.
It dates back to \cite{Eyring35, Kramers40} in the study of metastability in chemical reactions, and is proved rigorously by \cite{Bovier04, Bovier05}.
Lemma \ref{fact:EK} is the LSI version of the Eyring-Kramers law, which is stronger than the spectral gap estimate implied by the Poincar\'e inequality.

\quad Further define the Wasserstein distance $W_2(\mu, \nu)$ between $\mu$ and $\nu$ by
\begin{equation}
\label{eq:W2}
W_2(\mu, \nu): = \inf_{\Pi}\sqrt{ \int |x - y|^2 \Pi(dx,dy)},
\end{equation}
where the infimum is over all joint distributions $\Pi$ coupling $\mu$ and $\nu$.
We say that the probability measure $\nu$ satisfies Talagrand's inequality with constant $\gamma > 0$, if for all probability measure $\mu$ with $H(\mu|\nu)< \infty$,
\begin{equation}
\label{eq:Talagrand}
W_2(\mu, \nu) \le \frac{2}{\gamma}H(\mu|\nu).
\end{equation}
It follows from \cite[Theorem 1]{OV00} that the LSI implies Talagrand's inequality with the same constant.
That is,
if $\nu$ satisfies the LSI with constant $\alpha > 0$, then $\nu$ also satisfies Talagrand's inequality with constant $\gamma = \alpha$.
Combining with Lemma \ref{fact:EK}, we get a lower bound estimate of Talagrand's inequality constant for the Gibbs measure $\nu_{\tau}$.
\begin{lemma}
\label{fact:Tala}
Let $f$ satisfy Assumption \ref{assump:reggrow} $\&$ \ref{assump:nondeg}.
Then the Gibbs measure $\nu_{\tau}$ defined by \eqref{eq:Gibbs} satisfies Talagrand's inequality with constant $\gamma_{\tau} > 0$ such that
\begin{equation}
\label{eq:EKformulaTala}
\gamma_{\tau} \sim C \exp\left( -\frac{E_{*}}{\tau}\right) \quad \mbox{as } \tau \to 0
\end{equation}
where $C > 0$ depends on $f, d$.
\end{lemma}

\section{Continuous-time simulated annealing}
\label{sc4}

\quad In this section, we prove Theorem \ref{thm:contrate} by using the ideas developed in \cite{Miclo92,M18,MSTW18}.
Let $\mu_t$ be the probability measure of $X_t$ defined by \eqref{eq:SA}.
The key idea is to compare $\mu_t$ with the time-dependent Gibbs measure $\nu_{\tau_t}$ given by
\begin{equation}
\label{eq:Gibbstime}
\nu_{\tau_t}(dx) = \frac{1}{Z_{\tau_t}} \exp\left(-\frac{f(x)}{\tau_t} \right) dx,
\end{equation}
where $Z_{\tau_t}: = \int_{\mathbb{R}^d} \exp(-f(x)/\tau_t)$ is the normalizing constant.
Note that $\nu_{\tau_t}$ will concentrate on the minimum point of $f$ as $t \to \infty$ since $\tau_t \to 0$ as $t \to \infty$.
We will see that $\nu_{\tau_t}$ is close to $\mu_t$ in some sense as $t \to \infty$.
The proof of Theorem \ref{thm:contrate} is broken into four steps.

\medskip
{\bf Step 1: Reduce $\mu_t$ to $\nu_{\tau_t}$.}
We establish a bound that relates $\nu_{\tau_t}$ to $\mu_t$.
Let $(\widetilde{X}_t; \, t \ge 0)$ be a process whose distribution is $\nu_{\tau_t}$ at time $t$, coupled with $(X_t; \, t \ge 0)$ on the same probability space.
Fix $\delta > 0$. We have
\begin{align}
\label{eq:Pbound}
\mathbb{P}(f(X_t) > \delta) &= \mathbb{P}(f(X_t) > \delta, f(\widetilde{X}_t)> \delta) + \mathbb{P}(f(X_t) > \delta, f(\widetilde{X}_t) \le \delta) \notag\\
& \le \mathbb{P}(f(\widetilde{X}_t)> \delta) + || X_t - \widetilde{X}_t||_{TV} \notag \\
& \le \mathbb{P}(f(\widetilde{X}_t)> \delta) + \sqrt{2 H(\mu_t|\nu_{\tau_t})},
\end{align}
where we use Pinsker's inequality \cite[Lemma 2.5]{Tsy09} in the last inequation.
Now the problem boils down to estimating $\mathbb{P}(f(\widetilde{X}_t)> \delta)$ and $H(\mu_t|\nu_{\tau_t})$.

\medskip
{\bf Step 2: Long-time behavior of $f(\widetilde{X}_t)$.}
We study the asymptotics of $\mathbb{P}(f(\widetilde{X}_t)> \delta)$ as $t \to \infty$.
The following lemma provides a quantitative estimate of how $\nu_{\tau_t}$, or equivalently $\widetilde{X}_t$ concentrates on the minimum point of $f$ as $t \to \infty$.
\begin{lemma}
\label{lem:firterm}
Let $f$ satisfy Assumption \ref{assump:reggrow} $\&$ \ref{assump:nondeg}.
Assume that $\tau_t \sim \frac{E}{\ln t}$ as $t \to \infty$ with $E > E_{*}$.
For each $\varepsilon \in (0, \delta)$, there exist $C > 0$ (depending on $\delta, \varepsilon, f, E, d$) such that
\begin{equation}
\label{eq:nuest}
\mathbb{P}(f(\widetilde{X}_t)> \delta) \le C t^{-\frac{\delta - \varepsilon}{E}}.
\end{equation}
\end{lemma}
\begin{proof}
Note that
\begin{equation}
\label{eq:firtermdef}
\mathbb{P}(f(\widetilde{X}_t)> \delta) = \frac{\int_{f(x) > \delta} \exp(-f(x)/\tau_t) dx}{\int_{\mathbb{R}^d} \exp(-f(x)/\tau_t) dx}.
\end{equation}
Under Assumption \ref{assump:reggrow}, $f$ has quadratic growth, so at least linear growth at infinity \cite[Lemma 3.14]{MS14}: there exists $C > 0$ such that for $R$ large enough,
\begin{equation*}
f(x) \ge \min_{|z| = R} f(z) + C(|x| - R) \quad \mbox{for } |x| > R.
\end{equation*}
We can also choose $R$ sufficiently large so that $\min_{|z| = R} f(z) > \delta$.
Consequently,
\begin{align}
\label{eq:firtermnum}
\int_{f(x) > \delta} \exp(-f(x)/\tau_t) dx & = \int_{f(x) > \delta, |x| \le R} \exp(-f(x)/\tau_t) dx + \int_{f(x) > \delta, |x| > R}\exp(-f(x)/\tau_t) dx \notag \\
& = e^{-\frac{\delta}{\tau_t}} \vol(B_R) (1 + \mathcal{O}( \tau_t)),
\end{align}
where $\vol(B_R)$ is the volume of a ball with radius $R$.
Further by Laplace's method,
\begin{equation}
\label{eq:firtermdenom}
\int_{\mathbb{R}^d} \exp(-f(x)/\tau_t) dx \sim C (\tau_t)^{\frac{d}{2}}.
\end{equation}
By injecting \eqref{eq:firtermnum}, \eqref{eq:firtermdenom} into \eqref{eq:firtermdef},we get
\begin{equation}
\mathbb{P}(f(\widetilde{X}_t)> \delta) \le C t^{-\frac{\delta}{E}} (\ln t)^{\frac{d}{2}},
\end{equation}
which clearly yields \eqref{eq:nuest} since $(\ln t)^{\frac{d}{2}}/t^{\frac{\varepsilon}{E}} \to 0$ as $t \to \infty$.
\end{proof}

\begin{remark}
\label{rk:Laplace}
It is interesting to get a bound for $\mathbb{P}(f(\widetilde{X}_t) > \delta)$ when the dimension $d$ is large.
As mentioned in the introduction, the Laplace bound \eqref{eq:firtermdenom} may fail when $d, t \to \infty$ simultaneously.
Recall that $m_0$ is the minimum point of $f$.
By continuity of $f$, there exists $r > 0$ such that $f(x) < \varepsilon$ when $|x - m_0| < r$.
Thus,
\begin{align}
\label{eq:firtermdenom2}
\int_{\mathbb{R}^d} \exp(-f(x)/\tau_t) dx & \ge \int_{|x - m_0| < r} \exp(-f(x)/\tau_t) dx \notag \\
& \ge e^{-\frac{\varepsilon}{\tau_t}} \vol(B_r).
\end{align}
Further, if $t/e^{\frac{Ed}{CR}} \to \infty$ as $d \to \infty$,
\begin{equation}
\label{eq:firtermnum2}
\int_{f(x) > \delta} \exp(-f(x)/\tau_t) dx = e^{-\frac{\delta}{\tau_t}} \vol(B_R) (1 + \mathcal{O}( \tau_t d)),
\end{equation}
Combining \eqref{eq:firtermdenom2} and \eqref{eq:firtermnum2}, we get
\begin{equation}
\label{eq:nuest22}
\mathbb{P}(f(\widetilde{X}_t) > \delta) \le C \gamma^d t^{-\frac{\delta - \varepsilon}{E}},
\end{equation}
where $C > 0$ depends on $\delta, \varepsilon, f, E$, and $\gamma = \max(R/r, e^{\frac{\delta - \varepsilon}{CR}})$.
Also note that \cite{RRT17} obtained the bound $\mathbb{E}f(\widetilde{X}_t) \le C d/\ln t$.
By Markov's inequality, we get
\begin{equation}
\label{eq:boundRRT}
\mathbb{P}(f(\widetilde{X}_t) > \delta) \le C\delta^{-1} d \, (\ln t)^{-1}.
\end{equation}
In comparison with \eqref{eq:boundRRT}, the bound \eqref{eq:nuest22} is better in `$t$' but worse in `$d$'.
In terms of relaxation time, i.e. letting $\mathbb{P}(f(\widetilde{X}_t) > \delta)$ be of constant order,
both estimates show an exponential dependence of $t$ on $d$.
This suggests that SA is exponentially slow as the dimension increases.
\end{remark}

{\bf Step 3: Differential inequality for $H(\mu_t|\nu_{\tau_t})$.}
To get an estimate of $H(\mu_t|\nu_{\tau_t})$, we need to consider the time derivative $\frac{d}{dt}H(\mu_t|\nu_{\tau_t})$.
The following lemma is a reformulation of \cite[Proposition 3]{Miclo92}.
For ease of reference, we give a simplified proof here. First let us convent some notation. For an absolutely continuous measure $\mu(dx)$, we abuse the notation $\mu(dx) = \mu(x) dx$, i.e. $\mu(x)$ is the density of $\mu(dx)$.
So for two such probability measures $\mu$ and $\nu$, the Radon-Nikodym derivative $\frac{d\mu}{d\nu}(x)$ is identified with $\frac{\mu(x)}{\nu(x)}$.
\begin{lemma}
\label{lem:DI}
Let $\tau_t$ be decreasing in $t$. We have
\begin{equation}
\label{eq:diffineq}
\frac{d}{dt}H(\mu_t|\nu_{\tau_t}) \le -2 \tau_t I\left(\mu_t|\nu_{\tau_t}\right) + \frac{d}{dt}\bigg(\frac{1}{\tau_t} \bigg)\, \mathbb{E}f(X_t).
\end{equation}
where $I(\mu_t|\nu_{\tau_t})$ is the Fisher information defined by \eqref{eq:FI}.
\end{lemma}
\begin{proof}
Observe that
\begin{align}
\label{eq:DE}
\frac{d}{dt}H(\mu_t|\nu_{\tau_t}) & = \frac{d}{dt} \int \mu_t \ln\bigg(\frac{\mu_t}{\nu_{\tau_t}} \bigg) dx \notag\\
& = \underbrace{\int \frac{\partial \mu_t}{\partial t} \ln\bigg(\frac{\mu_t}{\nu_{\tau_t}} \bigg) dx}_{(a)} +\underbrace{\int \mu_t \frac{\partial}{\partial t} \bigg(\ln\bigg(\frac{\mu_t}{\nu_{\tau_t}} \bigg)\bigg) dx}_{(b)}.
\end{align}
We first consider the term (a).
Recall that $\mu_t$ satisfies the Fokker-Planck equation \eqref{eq:FPSA}.
Together with the fact that $\nabla(\tau_t \nu_{\tau_t}) = - \nu_{\tau_t} \nabla f$, we have
\begin{equation}
\label{eq:FPSA2}
\frac{\partial \mu_t}{\partial t} = \nabla \cdot \bigg(\tau_t \nu_{\tau_t} \nabla\bigg(\frac{\mu_t}{\nu_{\tau_t}} \bigg) \bigg).
\end{equation}
By injecting \eqref{eq:FPSA2} into the term (a) and further by integration by parts, we get
\begin{align}
\label{eq:terma}
(a) & = \int \nabla \cdot \bigg(\tau_t \nu_{\tau_t} \nabla\bigg(\frac{\mu_t}{\nu_{\tau_t}} \bigg) \bigg) \ln\bigg(\frac{\mu_t}{\nu_{\tau_t}} \bigg) dx \notag\\
& = - \int \tau_t \nu_{\tau_t} \nabla\bigg(\frac{\mu_t}{\nu_{\tau_t}} \bigg) \cdot \nabla \ln\bigg(\frac{\mu_t}{\nu_{\tau_t}} \bigg) dx \notag\\
& = - \tau_t \int \bigg| \nabla\bigg(\frac{\mu_t}{\nu_{\tau_t}} \bigg)\bigg|^2 \bigg(\frac{\mu_t}{\nu_{\tau_t}} \bigg)^{-1} d \nu_{\tau_t}
= -2 \tau_t I(\mu_t| \nu_{\tau_t}).
\end{align}
Now we consider the term (b). Direct computation leads to
\begin{align}
\label{eq:termb}
(b) = \int \bigg(\frac{\partial \mu_t}{\partial t} - \frac{\mu_t}{\nu_{\tau_t}} \frac{\partial \nu_{\tau_t}}{\partial t} \bigg)dx &= - \int \frac{\partial}{\partial t} \left(\ln \nu_{\tau_t} \right) d\mu_t \notag \\
& = \int \frac{d}{dt}\left(\ln Z_{\tau_t}\right) d\mu_t + \frac{d}{dt}\bigg(\frac{1}{\tau_t} \bigg) \mathbb{E}f(X_t) \notag\\
& \le \frac{d}{dt}\bigg(\frac{1}{\tau_t} \bigg) \, \mathbb{E}f(X_t),
\end{align}
where we use the fact that $\int \mu_t dx = 1$ in the second equation, and that $\tau_t$ is decreasing in $t$ in the last inequality.
Combining \eqref{eq:DE} with \eqref{eq:terma} and \eqref{eq:termb} yields \eqref{eq:diffineq}.
\end{proof}

{\bf Step 4: Estimating $H(\mu_t|\nu_{\tau_t})$ via the Erying-Kramers law.}
Note that there are two terms on the right hand side of \eqref{eq:diffineq}.
We start with an estimate of the second term.
\begin{lemma}
\label{lem:expfX}
Let $f$ satisfy Assumption \ref{assump:reggrow}, and assume that the condition \eqref{eq:moment} for $\mu_0$ holds.
For each $\varepsilon > 0$, there exists $C > 0$ (depending on $\varepsilon, f$) such that
\begin{equation}
\label{eq:expfX}
\mathbb{E}f(X_t) \le C(1+t)^{\varepsilon}.
\end{equation}
\end{lemma}
\begin{proof}
It is easy to see that Assumption \ref{assump:reggrow} implies Assumption $\mbox{H}_1$ in \cite{Miclo92}.
Together with the moment condition \eqref{eq:moment}, the proof follows the line of reasoning in \cite[Lemma 2]{Miclo92}.
\end{proof}

\quad Now we apply the Eyring-Kramers law, combining with a Gr\"{o}nwall-type argument to bound $H(\mu_t|\nu_{\tau_t})$ for large $t$.

\begin{lemma}
\label{lem:estH}
Let $f$ satisfy Assumption \ref{assump:reggrow} $\&$ \ref{assump:nondeg}, and assume that the condition \eqref{eq:moment} for $\mu_0$ holds.
Assume that $\tau_t \sim \frac{E}{\ln t}$ and $\frac{d}{dt}\left(\frac{1}{\tau_t}\right) = \mathcal{O}\left(\frac{1}{t}\right)$ as $t \to \infty$ with $E > E_{*}$.
For each $\varepsilon > 0$, there exists $C > 0$ (depending on $\varepsilon, f, E, d$) such that
\begin{equation}
\label{eq:estH}
H(\mu_t| \nu_{\tau_t}) \le Ct^{-\left(1 - \frac{E_{*}}{E} - \varepsilon \right)}.
\end{equation}
\end{lemma}
\begin{proof}
Using Lemma \ref{fact:EK} and the bound \eqref{eq:diffineq}, we have
\begin{equation}
\label{eq:dH}
\frac{d}{dt} H(\mu_t|\nu_{\tau_t}) \le - 2 \tau_t \alpha_t H(\mu_t|\nu_{\tau_t}) + \frac{C}{t} \, \mathbb{E}f(X_t),
\end{equation}
where $\alpha_t$ is the LSI constant for the Gibbs measure $\nu_{\tau_t}$.
By the Eyring-Kramers formula \eqref{eq:EKformula}, for each $\varepsilon > 0$,
there exist $C > 0$ and $t_0 > 0$,
\begin{equation}
\label{eq:alphat}
2 \tau_t \alpha_t \ge C t^{-\left( \frac{E_{*}}{E} - \varepsilon \right)} \quad \mbox{for } t \ge t_0.
\end{equation}
Combining \eqref{eq:dH} with \eqref{eq:expfX}, \eqref{eq:alphat}, we get
\begin{equation}
\label{eq:diffEK}
\frac{d}{dt} H(\mu_t|\nu_{\tau_t}) \le -C t^{-\left( \frac{E_{*}}{E} - \varepsilon \right)} H(\mu_t|\nu_{\tau_t}) + C't^{-1 + \varepsilon}.
\end{equation}
Fix $\varepsilon \in (0, \frac{1}{2} - \frac{E_{*}}{2E})$, let
\begin{equation*}
Q(t): = H(\mu_t|\nu_{\tau_t}) - \frac{2C'}{C}t^{-1 + \frac{E_{*}}{E}+ 2 \varepsilon}.
\end{equation*}
Then for $t_0$ sufficiently large and $t \ge  t_0$,
we have $\frac{d}{dt} Q(t) \le -Ct^{-\frac{E_{*}}{E} + \varepsilon} Q(t)$ by \eqref{eq:diffEK}.
This implies that $Q(t) \le Q(t_0) e^{-C \int_{t_0}^t s^{-\frac{E_{*}}{E} + \varepsilon} ds}$.
Thus,
\begin{equation}
\label{eq:finalest}
H(\mu_t|\nu_{\tau_t}) \le \frac{2C'}{C}t^{-1+ \frac{E_{*}}{E} + 2 \varepsilon} + H(\mu_{t_0}|\nu_{t_0}) e^{-\frac{C}{\kappa}(t^{\kappa} - t_0^{\kappa})},
\end{equation}
where $\kappa: = 1 - \frac{E_{*}}{E} - \varepsilon > 0$.
Note that the first term on the right hand side of \eqref{eq:finalest} dominates, and the conclusion follows.
\end{proof}

\quad Finally, by injecting \eqref{eq:nuest}, \eqref{eq:estH} into \eqref{eq:Pbound}, we get the desired estimate \eqref{eq:main1}.

\section{Discrete simulated annealing}
\label{sc5}

\quad This section is devoted to the proof of Theorem \ref{thm:discreterate}.
The idea is close to that employed for the continuous-time SA process \eqref{eq:SA}.
However, the analysis is more complicated due to discretization, and additional tools from \cite{VW19} on the convergence of ULA are used.
Recall that $\eta_k$ is the step size at iteration $k$, and $\Theta_k:=\sum_{j \le k} \eta_j$ is the cumulative step size up to iteration $k$.
Let $\mu_k$ be the probability density of $x_k$ defined by \eqref{eq:discreteSA}, and
\begin{equation}
\label{eq:Gibbsdistime}
\nu_{\tau_{\scaleto{\Theta_k}{4 pt}}}(dx) = \frac{1}{Z_{\tau_{\scaleto{\Theta_k}{4 pt}}}} \exp\left(-\frac{f(x)}{\tau_{\scaleto{\Theta_k}{5 pt}}} \right) dx,
\end{equation}
where $Z_{\tau_{\scaleto{\Theta_k}{4 pt}}}: = \int_{\mathbb{R}^d} \exp(-f(x)/\tau_{\scaleto{\Theta_k}{5 pt}}) dx$ is the normalizing constant.
We divide the proof into four steps.

\medskip
{\bf Step 1: Reduce $\mu_k$ to $\nu_{\tau_{\scaleto{\Theta_k}{4 pt}}}$.}
This step is similar to Step $1$ $\&$ $2$ in Section \ref{sc4}.
Let $(\widetilde{x}_k; \, k \ge 0)$ be a sequence whose distribution is $\nu_{\tau_{\scaleto{\Theta_k}{4 pt}}}$ at epoch $k$, coupled with $(x_k; \, k \ge 0)$ on the same probability space.
Fix $\delta > 0$.
The same argument as in \eqref{eq:Pbound} shows that
\begin{equation}
\label{eq:Pbounddis}
\mathbb{P}(f(x_k) > \delta) \le \mathbb{P}(f(\widetilde{x}_k) > \delta) + \sqrt{2 H(\mu_k|\nu_{\tau_{\scaleto{\Theta_k}{4 pt}}})}.
\end{equation}
Assume that $\Theta_k \to \infty$ and $\tau_{\scaleto{\Theta_k}{5 pt}} \sim \frac{E}{\ln \Theta_k}$ as $k \to \infty$ with $E > E_{*}$.
By Lemma \ref{lem:firterm}, we get a bound for the first term on the right hand side of \eqref{eq:Pbounddis}.
That is, for each $\varepsilon \in (0, \delta)$, there exists $C>0$ (depending on $\varepsilon, \delta, f, E, d$) such that
\begin{equation}
\label{eq:nuestdiscrete}
\mathbb{P}(f(\widetilde{x}_k) > \delta) \le C \Theta_k^{-\frac{\delta - \varepsilon}{E}}.
\end{equation}
So it remains to estimate $H(\mu_k|\nu_{\tau_{\scaleto{\Theta_k}{4 pt}}})$, which is the task of the next three steps.

\medskip
{\bf Step 2: Continuous-time coupling.}
To make use of continuous-time tools, we couple the sequence $(x_k; \, k \ge 0)$ by a continuous-time process $(X_t; \, t \ge 0)$ such that
$(X_{\Theta_k}; \, k \ge 0)$ has the same distribution as $(x_k; \, k \ge 0)$.
To do this, define the process $X$ by
\begin{equation}
\label{eq:couple}
dX_t = -\nabla f(x_k) dt + \sqrt{2 \tau_{\scaleto{\Theta_k}{5 pt}}} dB_t, \quad t \in [\Theta_k, \Theta_{k+1}),
\end{equation}
where we identify $X_{\Theta_k}$ with $x_k$.
Note that the Fokker-Planck equation \eqref{eq:FPSA} plays an important role in the analysis of continuous-time SA \eqref{eq:SA}.
It is desirable to get a version of the Fokker-Planck equation for the coupled process \eqref{eq:couple}.
The result is stated as follows.
\begin{lemma}
For $t \in [\Theta_k, \Theta_{k+1})$, the probability density $\mu_t$ of $X_t$ defined by \eqref{eq:couple} satisfies the following equation:
\begin{equation}
\label{eq:FPcouple}
\frac{\partial \mu_t}{\partial t} = \nabla \cdot \bigg(\tau_{\scaleto{\Theta_k}{5 pt}} \nu_{\tau_{\scaleto{\Theta_k}{4 pt}}} \nabla \bigg(\frac{\mu_t}{\nu_{\tau_{\scaleto{\Theta_k}{4 pt}}}} \bigg) \bigg)
+ \nabla \cdot \left(\mu_t \, \mathbb{E}[\nabla f(x_k) - \nabla f(X_t)| X_t = x] \right).
\end{equation}
\end{lemma}
\begin{proof}
Let $\mu_{t|s}(x|y)$ be the conditional probability $\mathbb{P}(X_t = x|X_s = y)$.
By conditioning on $X_{\Theta_k} = x_k$, we have
\begin{equation}
\label{eq:conditionFP}
\frac{\partial \mu_{t|\Theta_k}(x|x_k)}{\partial t} = \nabla \cdot (\mu_{t|\Theta_k}(x|x_k) \nabla f(x_k)) + \tau_{\Theta_k} \Delta \mu_{t|\Theta_k}(x|x_k).
\end{equation}
By integrating \eqref{eq:conditionFP} against $\mu_{\Theta_k}$, and using the fact that $\mu_{t|\Theta_k}(x|x_k) \mu_{\Theta_k}(x_k) = \mu_t(x) \mu_{\Theta_k|t}(x_k|x)$, we get
\begin{equation}
\label{eq:FPcouple2}
\frac{\partial \mu_t}{\partial t} = \nabla \cdot (\mu_t(x) \, \mathbb{E}[\nabla f(x_k)|X_t = x]) + \tau_{\scaleto{\Theta_k}{5 pt}} \Delta \mu_t.
\end{equation}
Further by \eqref{eq:FPSA2}, we have
\begin{equation}
\label{eq:FPSA3}
\nabla \cdot \bigg(\tau_{\scaleto{\Theta_k}{5 pt}} \nu_{\tau_{\scaleto{\Theta_k}{4 pt}}} \nabla \bigg(\frac{\mu_t}{\nu_{\tau_{\scaleto{\Theta_k}{4 pt}}}} \bigg) \bigg) = \nabla \cdot (\mu_t \nabla f(x)) + \tau_{\scaleto{\Theta_k}{5 pt}} \Delta \mu_t.
\end{equation}
Combining \eqref{eq:FPcouple2} and \eqref{eq:FPSA3} yields \eqref{eq:FPcouple}.
\end{proof}

\quad There are two terms on the right hand side of \eqref{eq:FPcouple}.
Comparing to \eqref{eq:FPSA2}, the first term is the usual Fokker-Planck term, while the second term corresponds to the discretization error.

\medskip
{\bf Step 3: One-step analysis of $H(\mu_k|\nu_{\tau_{\scaleto{\Theta_k}{4 pt}}})$.}
Here we use the coupled process \eqref{eq:couple} to study the one-step decay of $H(\mu_k|\nu_{\tau_{\scaleto{\Theta_k}{4 pt}}})$.

\begin{lemma}
\label{lem:onestep}
Let $f$ satisfy Assumption \ref{assump:reggrow}, \ref{assump:nondeg} $\&$ \ref{assump:upper}, and assume that
the condition \eqref{eq:moment} for $\mu_0$ holds.
Assume that $\tau_t \sim \frac{E}{\ln t}$ and $\frac{d}{dt}\left(\frac{1}{\tau_t}\right) = \mathcal{O}\left(\frac{1}{t}\right)$ as $t \to \infty$ with $E > E_{*}$.
Also assume that  $\Theta_k \to \infty$ and $\eta_{k+1}\Theta_k \to 0$ as $k \to \infty$.
Then, for each $\varepsilon > 0$, there exist $C, C' > 0$ (depending on $\varepsilon, f, E, d$) such that
\begin{align}
\label{eq:keyonestep}
H(\mu_{k+1}|\nu_{\tau_{\scaleto{\Theta_{k+1}}{4 pt}}})  & \le
 \bigg(1 - C\eta_{k+1}\Theta_k^{-(\frac{E_{*}}{E} - \varepsilon)}\bigg)H(\mu_{k}|\nu_{\tau_{\scaleto{\Theta_k}{4 pt}}}) \notag \\
&  \qquad \qquad \qquad \quad+ C'(\eta_{k+1}^2+ \eta_{k+1}^3 \ln \Theta_k +  \eta_{k+1} \Theta_k^{-1 + \varepsilon}).
\end{align}
\end{lemma}
\begin{proof}
Write
\begin{equation}
\label{eq:decomp}
H(\mu_{k+1}|\nu_{\tau_{\scaleto{\Theta_{k+1}}{4 pt}}}) =  \underbrace{H(\mu_{k+1}|\nu_{\tau_{\scaleto{\Theta_{k}}{4 pt}}})}_{(a)} + \underbrace{(H(\mu_{k+1}|\nu_{\tau_{\scaleto{\Theta_{k+1}}{4 pt}}}) - H(\mu_{k+1}|\nu_{\tau_{\scaleto{\Theta_{k}}{4 pt}}}))}_{(b)}.
\end{equation}
We first use the coupled process \eqref{eq:couple} to study the term $(a)$.
Note that
\begin{align}
\label{eq:timedecomp}
\frac{d}{dt} H(\mu_t | \nu_{\tau_{\scaleto{\Theta_k}{4 pt}}}) & = \int \frac{\partial \mu_t}{\partial t} \ln \bigg(\frac{\mu_t}{\nu_{\tau_{\scaleto{\Theta_k}{4 pt}}}} \bigg)dx + \int \mu_t \frac{d}{dt}\ln \bigg(\frac{\mu_t}{\nu_{\tau_{\scaleto{\Theta_k}{4 pt}}}} \bigg) dx \notag\\
&  = \int  \nabla \cdot \bigg(\tau_{\scaleto{\Theta_k}{5 pt}} \nu_{\tau_{\scaleto{\Theta_k}{4 pt}}} \nabla \bigg(\frac{\mu_t}{\nu_{\tau_{\scaleto{\Theta_k}{4 pt}}}} \bigg) \bigg)  \ln \bigg(\frac{\mu_t}{\nu_{\tau_{\scaleto{\Theta_k}{4 pt}}}} \bigg)dx \notag \\
&\quad  + \underbrace{\int  \nabla \cdot \left(\mu_t \, \mathbb{E}[\nabla f(x_k) - \nabla f(X_t)| X_t = x] \right)\ln \bigg(\frac{\mu_t}{\nu_{\tau_{\scaleto{\Theta_k}{4 pt}}}} \bigg)dx}_{(c)} + \frac{d}{dt}\int \mu_t(dx) \notag \\
& = - 2 \tau_{\scaleto{\Theta_k}{5 pt}} I(\mu_t|\nu_{\tau_{\scaleto{\Theta_k}{4 pt}}}) + (c),
\end{align}
where we use \eqref{eq:FPcouple} in the second equation, and \eqref{eq:terma} in the third equation.
Now we need to estimate the term $(c)$ in \eqref{eq:timedecomp}.
By integration by parts and the fact that $a \cdot b \le \frac{1}{\tau_{\scaleto{\Theta_k}{4 pt}}} |a|^2 + \frac{\tau_{\scaleto{\Theta_k}{4 pt}}}{4} |b|^2$, we get
\begin{align}
\label{eq:discreterr}
(c) & = \mathbb{E}\bigg((\nabla f(X_t) -\nabla f(x_k)) \cdot \nabla \log \bigg(\frac{\mu_t}{\nu_{\tau_{\scaleto{\Theta_k}{4 pt}}}} \bigg) \bigg) \notag \\
& \le \frac{1}{\tau_{\scaleto{\Theta_k}{5 pt}}} \mathbb{E}|\nabla f(X_t) -\nabla f(x_k))|^2 + \frac{\tau_{\scaleto{\Theta_k}{5 pt}}}{4} \mathbb{E} \bigg| \log \bigg(\frac{\mu_t}{\nu_{\tau_{\scaleto{\Theta_k}{4 pt}}}} \bigg)\bigg|^2 \notag\\
& \le \frac{L^2}{\tau_{\scaleto{\Theta_k}{5 pt}}} \mathbb{E}|X_t - x_k|^2 + \frac{\tau_{\scaleto{\Theta_k}{5 pt}}}{2} I(\mu_t|\nu_{\tau_{\scaleto{\Theta_k}{4 pt}}}),
\end{align}
where $L$ is the Lipschitz constant of $\nabla f$ by Assumption \ref{assump:upper}.
Recall from \eqref{eq:couple} that $X_t - x_k = - \nabla f(x_k) (t - \Theta_k) + \sqrt{2 \tau_{\scaleto{\Theta_k}{5 pt}} (t - \Theta_k)} Z$, where $Z$ is standard normal.
Consequently,
\begin{align}
\label{eq:diffonestep}
\mathbb{E}|X_t - x_k|^2 &= (t - x_k)^2 \mathbb{E}|\nabla f(x_k)|^2 + 2 \tau_{\scaleto{\Theta_k}{5 pt}} (t - x_k) d \notag \\
& \le \eta_{k+1}^2 \mathbb{E}|\nabla f(x_k)|^2 + C \tau_{\scaleto{\Theta_k}{5 pt}} \eta_{k+1}.
\end{align}
According to Lemma \ref{fact:Tala}, $\nu_{\tau_{\scaleto{\Theta_k}{4 pt}}}$ satisfies Talagrand's inequality with constant $\gamma_{\tau_{\scaleto{\Theta_k}{4 pt}}} \sim \kappa \exp(-E_{*}/\tau_{\scaleto{\Theta_k}{5 pt}})$.
Further by \cite[Lemma 10]{VW19},
\begin{equation}
\label{eq:gradest}
\mathbb{E}|\nabla f(x_k)|^2 \le  \frac{C}{\gamma_{\tau_{\scaleto{\Theta_k}{4 pt}}}}H(\mu_k|\nu_{\tau_{\scaleto{\Theta_k}{4 pt}}}) + C.
\end{equation}
Combining \eqref{eq:discreterr} with \eqref{eq:diffonestep}, \eqref{eq:gradest} and the fact that $\tau_{\scaleto{\Theta_k}{5 pt}} \sim \frac{E}{\ln \Theta_k}$ as $k \to \infty$, we have
\begin{equation}
\label{eq:cest}
(c) \le C \bigg(\eta_{k+1}^2 \Theta_k^{\frac{E_{*}}{E}} \ln \Theta_k \bigg)  H(\mu_k|\nu_{\tau_{\scaleto{\Theta_k}{4 pt}}}) + C  (\eta_{k+1}+ \eta_{k+1}^2 \ln \Theta_k) +  \frac{\tau_{\scaleto{\Theta_k}{5 pt}}}{2} I(\mu_t|\nu_{\tau_{\scaleto{\Theta_k}{4 pt}}}).
\end{equation}
Injecting \eqref{eq:cest} into \eqref{eq:timedecomp} and further by Lemma \ref{fact:EK}, we get
\begin{align*}
\frac{d}{dt} H(\mu_t |\nu_{\tau_{\scaleto{\Theta_k}{4 pt}}}) &\le -\frac{3}{2}\tau_{\scaleto{\Theta_k}{5 pt}} I(\mu_t|\nu_{\tau_{\scaleto{\Theta_k}{4 pt}}}) + C' \bigg(\eta_{k+1}^2 \Theta_k^{\frac{E_{*}}{E}} \ln \Theta_k \bigg) H(\mu_k|\nu_{\tau_{\scaleto{\Theta_k}{4 pt}}}) + C' ( \eta_{k+1}+ \eta_{k+1}^2 \ln \Theta_k) \notag \\
& \le -\frac{3}{2} C \Theta_k^{-(\frac{E_{*}}{E} - \varepsilon)}H(\mu_t|\nu_{\tau_{\scaleto{\Theta_k}{4 pt}}}) + C' \bigg(\eta_{k+1}^2 \Theta_k^{\frac{E_{*}}{E}+ \varepsilon} H(\mu_k|\nu_{\tau_{\scaleto{\Theta_k}{4 pt}}}) + (\eta_{k+1}+ \eta_{k+1}^2 \ln \Theta_k) \bigg),
\end{align*}
Now by a Gr\"{o}nwall argument, we have
\begin{align}
\label{eq:decompfirst}
H(\mu_{k+1}|\nu_{\tau_{\scaleto{\Theta_k}{4 pt}}}) & \le e^{- \frac{3}{2}C \eta_{k+1} \Theta_k^{-(\frac{E_{*}}{E} - \varepsilon)}}\left( (1 + C' \eta_{k+1}^3 \Theta^{\frac{E_{*}}{E} + \varepsilon}) H(\mu_{k}|\nu_{\tau_{\scaleto{\Theta_k}{4 pt}}}) + C'(\eta_{k+1}^2+ \eta_{k+1}^3 \ln \Theta_k) \right) \notag \\
& \le e^{-\frac{5}{4}C \eta_{k+1}\Theta_k^{-(\frac{E_{*}}{E} - \varepsilon)}}H(\mu_{k}|\nu_{\tau_{\scaleto{\Theta_k}{4 pt}}})
+ C'(\eta_{k+1}^2+ \eta_{k+1}^3 \ln \Theta_k) \notag \\
& \le \bigg(1 - C\eta_{k+1}\Theta_k^{-(\frac{E_{*}}{E} - \varepsilon)}\bigg)H(\mu_{k}|\nu_{\tau_{\scaleto{\Theta_k}{4 pt}}})
+ C'(\eta_{k+1}^2+ \eta_{k+1}^3 \ln \Theta_k),
\end{align}
where we use the fact that $\eta_{k+1} \Theta_k^{\frac{E_{*}}{E}} \to 0$ as $k \to \infty$ in the second inequality.

\quad Now we consider the term $(b)$ in \eqref{eq:decomp}.
Note that
\begin{align}
\label{eq:diffHkk1}
H(\mu_{k+1}|\nu_{\tau_{\scaleto{\Theta_{k+1}}{4 pt}}}) - H(\mu_{k+1}|\nu_{\tau_{\scaleto{\Theta_{k}}{4 pt}}}) &= \ln \bigg(\frac{Z_{\tau_{\scaleto{\Theta_{k+1}}{4 pt}}}}{Z_{\tau_{\scaleto{\Theta_k}{4 pt}}}} \bigg) + \bigg( \frac{1}{\tau_{\scaleto{\Theta_{k+1}}{5 pt}}} - \frac{1}{\tau_{\scaleto{\Theta_k}{4 pt}}}\bigg) \mathbb{E}f(x_{k+1}) \notag \\
& \le C \frac{\eta_{k+1}}{\Theta_k} \mathbb{E}f(x_{k+1}).
\end{align}
We claim that for each $\varepsilon>0$, 
$\mathbb{E}f(x_{k+1}) \le C \le C \Theta_k^{\varepsilon}$.
Choose $C > 0$ sufficiently large, and let $\mathbb{E}f(x_{k+1})$ be the first term exceeding $C$.
By Assumption \ref{assump:upper},
\begin{align*}
f(x_{k+1}) \le f(x_k) - \eta_k |\nabla f(x_k)|^2 +\sqrt{2 \tau_{\scaleto{\Theta_k}{5 pt}}\eta_k} \nabla f(x_k) \cdot Z_k
+ \frac{L}{2} |\eta_k \nabla f(x_k) +\sqrt{2 \tau_{\scaleto{\Theta_k}{5 pt}}\eta_k} Z_k |^2
\end{align*}
Further by taking expectation, we get
\begin{equation}
\label{eq:diffkk1}
\mathbb{E}f(x_{k+1}) - \mathbb{E}f(x_k) \le - \eta_k \bigg(1 - \frac{\eta_k L}{2} \bigg) \mathbb{E}|\nabla f(x_k)|^2 + Ld \tau_{\scaleto{\Theta_k}{5 pt}}\eta_k.
\end{equation}
Thus, $\mathbb{E}f(x_{k+1}) - \mathbb{E}f(x_k) \le Ld \tau_{\scaleto{\Theta_k}{5 pt}}\eta_k$ which implies that $\mathbb{E}f(x_k) > C-1$ for $k$ large enough.
By Assumption \ref{assump:reggrow}(ii),
$\mathbb{E}|\nabla f(x_k)|^2 > C'$ for some $C' > 0$.
Combining with \eqref{eq:diffkk1}, we have $\mathbb{E}f(x_k) > \mathbb{E}f(x_{k+1}) \ge C$ for $k$ large enough.
This contradicts the fact that $\mathbb{E}f(x_{k+1})$ is the first term exceeding $C$.
Now by \eqref{eq:diffHkk1}, we get
\begin{equation}
\label{eq:decompsecond}
H(\mu_{k+1}|\nu_{\tau_{\scaleto{\Theta_{k+1}}{4 pt}}}) - H(\mu_{k+1}|\nu_{\tau_{\scaleto{\Theta_{k}}{4 pt}}}) \le C \eta_{k+1} \Theta_k^{-1 + \varepsilon}.
\end{equation}
Combining \eqref{eq:decomp} with \eqref{eq:decompfirst}, \eqref{eq:decompsecond} yields \eqref{eq:keyonestep}.
\end{proof}

{\bf Step 4: Estimating $H(\mu_k|\nu_{\tau_{\scaleto{\Theta_k}{4 pt}}})$.}
We use Lemma \ref{lem:onestep} to derive an estimate for $H(\mu_k|\nu_{\tau_{\scaleto{\Theta_k}{4 pt}}})$.
Under the condition \eqref{eq:dominate}, the term $\eta_{k+1} \Theta_k^{-1 + \varepsilon}$ dominates
$\eta_{k+1}^2$, $\eta_{k+1}^3 \ln \Theta_k$ as $k \to \infty$.
Thus, the recursion \eqref{eq:keyonestep} yields
\begin{equation*}
H(\mu_{k+1}|\nu_{\tau_{\scaleto{\Theta_{k+1}}{4 pt}}}) \le
 \bigg(1 - C\eta_{k+1}\Theta_k^{-(\frac{E_{*}}{E} - \varepsilon)}\bigg)H(\mu_{k}|\nu_{\tau_{\scaleto{\Theta_k}{4 pt}}}) + C' \eta_{k+1} \Theta_k^{-1 + \varepsilon}.
\end{equation*}
Since $E_{*}/E < 1$, a similar argument as in Lemma \ref{lem:estH} shows that
\begin{equation*}
H(\mu_{k+1}|\nu_{\tau_{\scaleto{\Theta_{k+1}}{4 pt}}}) - C \Theta_k^{-(1 - \frac{E_{*}}{E} - \varepsilon)} \le \bigg(1 - C' \eta_{k+1}\Theta_k^{-(\frac{E_{*}}{E} - \varepsilon)}\bigg) \bigg(H(\mu_{k}|\nu_{\tau_{\scaleto{\Theta_k}{4 pt}}}) - C \Theta_{k-1}^{-(1 - \frac{E_{*}}{E} - \varepsilon)}\bigg).
\end{equation*}
This together with the condition \eqref{eq:dominate2} implies that
\begin{equation}
\label{eq:estHdis}
H(\mu_{k}|\nu_{\tau_{\scaleto{\Theta_k}{4 pt}}})  \le C\Theta_{k}^{-(1 - \frac{E_{*}}{E} - \varepsilon)}.
\end{equation}
By injecting \eqref{eq:nuestdiscrete} and \eqref{eq:estHdis} into \eqref{eq:Pbounddis} we obtain \eqref{eq:main2}.

\section{Conclusion}
\label{sc6}

\quad In this paper, we study the convergence rate of SA in both continuous and discrete settings.
The main tool is functional inequalities for the Gibbs measure at low temperatures.
We prove that the tail probability, in both settings, exhibits a polynomial decay in time.
The decay rate is also given as function of the model parameters.
In the discrete setting, we derive a condition on the step size to ensure the convergence to the global minimum.
This condition may be useful in tuning the step size.

\quad There are  a few directions to extend this work.
For instance, one can study the convergence rate of SA for L\'evy flight with a suitable cooling schedule.
Another problem is to study the dependence of the convergence rate in the dimension $d$.
This requires a deep understanding of the Eyring-Kramers law in high dimension, and is related to the Laplace approximation of high dimensional integrals.
Both problems are worth exploring, but may be challenging.

\bigskip
{\bf Acknowledgement:}
Tang thanks Georg Menz for helpful discussions.
He also gratefully acknowledges financial support through a start-up grant at Columbia University. Zhou gratefully acknowledges financial supports through a start-up grant at Columbia University and through the Nie Center for Intelligent Asset Management.

\bibliographystyle{abbrvnat}
\bibliography{unique}
\end{document}